%% file: main.tex
\documentclass[3p,times,procedia]{elsarticle}
\usepackage{ecrc}
\usepackage[bookmarks=false]{hyperref}
    \hypersetup{colorlinks,
      linkcolor=blue,
      citecolor=blue,
      urlcolor=blue}
\volume{00}

\firstpage{1}

\journalname{Procedia Computer Science}

\runauth{Tudor Pistol}

\jid{procs}

\usepackage{amssymb}
\usepackage{amsthm}
\usepackage{amsmath}
\usepackage{subcaption}
\usepackage{algorithm}
\usepackage{algpseudocode}
\usepackage{booktabs} 
\usepackage{multirow}  
\usepackage{graphicx}
\usepackage[figuresright]{rotating}
\newtheorem{lemma}{Lemma}

\begin{document}
\begin{frontmatter}



\dochead{30th International Conference on Knowledge-Based and Intelligent Information \& Engineering Systems (KES 2026)}%
%% e.g. \dochead{17th International Conference on Dynamical Processes in Excited States of Solids}

\title{Learning Doubly Sparse Explicitly Conditioned Transforms}
\author[]{Tudor Pistol\corref{cor1}} 

\address{Department of Mathematics, Faculty of Mathematics and Computer Science, University of Bucharest, Academiei 14, Bucharest, 010014, Romania}


\input{sections/abstract}

\cortext[cor1]{Corresponding author. Tel.: +40-756-927-417}
\end{frontmatter}

%\correspondingauthor[*]{Corresponding author. Tel.: +0-000-000-0000 ; fax: +0-000-000-0000.}
\email{tudor.pistol@s.unibuc.ro}

\input{sections/0-introduction}
\input{sections/1-problem_formulation}
\input{sections/2-theoretical_analysis}
\input{sections/3-algorithm}
\input{sections/4-numerical_experiments}
\input{sections/5-conclusions}

\section*{Acknowledgements}

I would like to express my appreciation to Professor Paul Irofti for his invaluable guidance, as well as Professor Cristian Rusu for our stimulating conversations and his testing suggestions. I want to thank Professor Ionel Popescu for support and trust. 

I would like to express my gratitude to my family, and particularly to my mother, Liliana, for her unwavering support and encouragement. I also thank my friend, Teofil, for our insightful conversations and understanding. I am grateful to Liana for her unconditional patience and inspiration.

% Acknowledgements and Reference heading should be left justified, bold, with the first letter capitalized but have no numbers. Text below continues as normal.



% \appendix
% \section{An example appendix}
% Authors including an appendix section should do so before References section. Multiple appendices should all have headings in the style used above. They will automatically be ordered A, B, C etc.

% \subsection{Example of a sub-heading within an appendix}
% There is also the option to include a subheading within the Appendix if you wish.
% %% References
% %%
% %% Following citation commands can be used in the body text:
% %% Usage of \cite is as follows:
% %%   \cite{key}         ==>>  [#]
% %%   \cite[chap. 2]{key} ==>> [#, chap. 2]
% %%

% %The citation must be used in following style: \cite{article-minimal} \cite{article-full} \cite{article-crossref} \cite{whole-journal}.
% %% References with BibTeX database:

% %\bibliography{xampl}
% %\bibliographystyle{elsarticle-harv}
\bibliographystyle{elsarticle-harv}
\bibliography{references}

\end{document}

%% file: sections/abstract.tex
% !TEX root = ./main.tex

\begin{abstract}
Finding convenient spaces in which certain hypotheses regarding an assumed sparse structure of natural signals hold true has become a desirable result in recent research, its implications being reflected in areas such as data compression, noise reduction and feature extraction. While the extensively used analytical transforms, such as DFT or DCT, already provide efficient algorithms and robust sparse representations, they assume a fixed prior about the data, failing to accurately capture the specific structure of more restrictive classes of signals. To address this, the concept of a data-adaptive, learnt transform has been introduced in the literature, allowing for the reduction of a residual term in the transform domain. More recent studies have shown that the condition number serves as a good metric in this context, where the desired outcome alternates between a generalizing tendency and one that achieves minimal approximation error. Motivated by these considerations, we introduce the learning of a structured, explicitly conditioned transform formulated as the product of a fixed canonical matrix and a refining data-adaptive sparse component. This approach seeks to preserve the advantages of fast and stable analytical transforms, while introducing controllable adaptivity to the data. No references that concern this specific formulation have been identified so far, indicating its novelty. The proposed algorithm is motivated within the framework of inexact proximal methods, leveraging a newly derived closed-form projection operator. Empirical observations demonstrate state-of-the-art results on the doubly sparse transform learning problem and comparable performance with its dense variant at significantly lower computational costs and sometimes faster convergence and better avoidance of bad local minima.
\end{abstract}

%% file: sections/0-introduction.tex
% !TEX root = ./main.tex
\section{Introduction}
The problem of finding sparse representations in signal processing has been extensively studied due to their potential for capturing essential information. It has been observed that natural signals tend to have fewer essential components in specific domains, while the remaining ones are typically insignificant in terms of magnitude. Hence, such mappings are utilized in denoising applications and feature extraction for fields like image and audio processing, biomedical signal analysis and telecommunications. Furthermore, the usefulness of sparse representations is extended to data science, where information must be brought into a convenient form in order to be stored or transmitted efficiently. More recently, machine learning has introduced new demands, where sparse modeling plays an essential role in dimensionality reduction, classification, clustering, and model regularization. The latter helps prevent overfitting by suppressing noise and irrelevant patterns that lead to poor generalization. By eliminating redundant components or elements unrelated to the underlying signal, such as samples drawn from noise distributions, sparse modeling enhances the identification of the signal's intrinsic structure and improves its separability from other patterns and background data. 

\subsection{Transform Learning Problem}
The literature on sparse representation techniques distinguishes between two main approaches: the analytic approach and the learning-based approach. The first category introduces mathematically defined transforms, like the Discrete Cosine Transform (DCT), the Discrete Fourier Transform (DFT) and Wavelets, which are typically closed-form solutions derived from signal processing principles or assumptions about stationary stochastic processes. However, the assumptions underlying these solutions may be too general and might fail to capture the intrinsic structure of more specific classes of signals. Consequently, relaxing these assumptions is often considered necessary to allow for data adaptability.

The learning-based approach is widely discussed in the context of the dictionary learning problem~\cite{DL_book} and can be further divided into several subcategories, depending on the underlying assumptions. These models are often formulated as Maximum A Posteriori (MAP) estimators, assuming a certain prior distribution for the data and seeking to maximize the posterior probability given the observed distribution. This typically translates to minimizing the reconstruction error in a least-squares sense, implicitly assuming a mean-centered multivariate white Gaussian noise distribution. Two such popular models are synthesis and analysis priors~\cite{7071162}.

The transform learning problem was introduced by Ravishankar and Bresler in a series of papers~\cite{6339108, 7045534, 6466952, 6638226, 7069264, 7051209}. It relies on the prior assumption that a signal \( y \in \mathbb{R}^n \) is \textit{approximately sparsifiable} in a domain induced by the transform $ W \in \mathbb{R}^{n \times n} $, i.e., $ W y = x + \xi $ (where $ x $ is the assumed sparse signal and $\xi $ denotes the noise component), on the premise that the noise component lies in the transform domain, as opposed to the signal domain. This resembles the mechanism of the canonical transforms, which decompose a signal into a series of functions, each weighted by corresponding spectral coefficients. The transform learning optimization problem is stated as:
\begin{equation}
\min_{W \in \mathbb{R}^{n \times n},\, X \in \mathbb{R}^{n \times N}} \quad \| W Y - X \|_F^2 - \lambda \delta \log |\det W| + \lambda \|W\|_F \quad \text{s.t.} \quad \|X_i \|_0 \leq r,
\label{eq:bresler_2}
\end{equation}
where $ W \in \mathbb{R}^{n \times n} $ is the learnt transform, while the signals and their sparse representation are stacked as columns in $ Y $ and $ X \in \mathbb{R}^{n \times N}$.

To solve this objective, the general transform learning framework employs an alternating minimization scheme, iteratively alternating between a sparse coding step, learning the sparse representations $X$, for which a closed-form solution obtained by hard-thresholding is derived, and a transform update step to learn the matrix $W$. The problem, however, in this transform update step, is the conditioning of the transform matrix, which the original authors managed by incorporating logarithmic determinant and Frobenius norm penalty functions into the objective. In the recent work of Pătrașcu et al.~\cite{PRI24}, this behavior is analyzed through the learned transform's tendency to balance generalization and particularization. The condition number becomes a measure of the trade-off between generalization behavior and representation error. Thus, the optimal transform inherently interpolates between an orthogonal matrix, also known as the solution to the Orthogonal Procrustes Problem~\cite{Schönemann_1966}, and the unconstrained least-squares solution. The importance of transform conditioning demands more rigorous restrictions, therefore explicit constraints replace regularization-based approaches:
\begin{equation}
\min_{W \in \mathbb{R}^{n \times n},\, X \in \mathbb{R}^{n \times N}} \quad \| W Y - X \|_F^2 \quad \text{s.t.} \quad \|X_i \|_0 \leq r, \quad \kappa(W) \leq \rho, \quad \|W\|_F = \tau,
\label{eq:patrascu_1}
\end{equation}
where the equality constraint enforces the avoidance of trivial solutions over the singular spectrum and the real parameter $ \rho $ is chosen to be greater than $ 1 $, since the highest level of generalization is achieved by an orthogonal matrix.

\subsection{Double Sparsity}

Given the performance of the canonical transforms, it is of interest to extend their speed and stability to settings that allow for data-adaptive flexibility. The literature often considers the product $ T \Phi $, where $ T $ is the learnt factor and $ \Phi $ denotes a known analytical transform. Under this formulation, the flexible part aims to refine the space induced by $ \Phi $, enhancing the ability to effectively represent a set of signals. This refinement enables the assumption of sparsity over $ T $, as it acts on the already efficient transform, relying on its performance. Therefore, the problem has a doubly sparse significance: learning a sparse matrix for sparse signal representation. 

This exact problem has been approached for the dictionary learning problem~\cite{elad_doubly} and adapted to the transform learning framework in~\cite{6466952}. The doubly sparse transform learning problem is the same as Problem~\eqref{eq:bresler_2} with an additional sparsity constraint, i.e., $ \| T \|_0 \leq R $, where $ T $ is the learnt factor in the $ W = T \Phi $ decomposition and $ R $ denotes the maximum allowed number of non-zero elements. This heuristic has been proven to work well in practice.

The fast convergence promised by the explicit conditioning and the aforementioned ideas about double sparsity motivate the contribution of this paper.

\subsection{Contribution}

We approach the linear transformation used in the transform-based prior as the product $ T \Phi $, where $ \Phi $ represents an analytic orthogonal transform, such as the DFT or DCT, and $ T $ is the learnt, sparse part, with an explicit bound on its conditioning. The sparse factor aims to improve adaptability to data using a small number of refining coefficients, relying on the stability of the analytical transform. The connection between matrix components and spectral information remains, however, an open problem in the literature. We outline the main contributions of this work:
\begin{itemize}
    \item We introduce a doubly sparse, explicitly conditioned transform learning framework. Although it builds upon existing principles, this specific problem formulation is introduced for the first time in this paper.
    \item We develop a proximal-based optimization algorithm to solve a relaxed formulation. A key component of our approach is the singular spectrum projection, for which we derive a closed-form solution and prove its uniqueness despite the non-convexity of the feasible space. To the best of our knowledge, this exact projection has not been previously established in the literature.
    \item We discuss a theoretical convergence analysis, which suggests that the proposed method could be integrated into modern non-convex optimization frameworks.
    \item We empirically demonstrate that by incorporating specific algorithmic enhancements, the practical implementation significantly accelerates the convergence of the objective function and yields better results than existing approaches on certain metrics, particularly when the feasible space is more permissive.
\end{itemize}
\subsection{Paper Structure}

This paper is organized as follows. Section 2 introduces the \textit{doubly sparse explicitly conditioned transform learning} problem. Section 3 provides a theoretical analysis of the established framework, deriving the closed-form projection of the singular spectrum and discussing the convergence to a stationary point under specific conditions. Section 4 details the practical implementation of the proposed algorithm, introduces specific algorithmic enhancements, and analyzes its computational complexity, while numerical experiments are presented in Section 5. We conclude in Section 6.

%% file: sections/1-problem_formulation.tex
% !TEX root = ./main.tex

\section{Doubly Sparse Explicitly Conditioned Transform Learning Problem}

Let $ Y \in \mathbb{R}^{n \times N} $ and $ X \in \mathbb{R}^{n \times N} $ denote the data matrix and its corresponding representation at a given iteration of the alternating minimization algorithm, where $ n $ is the dimension of each of the $ N $ signals in the given set. Consider the aforementioned matrix decomposition $ W = T \Phi $ with $ W$, $ T $, $\Phi \in \mathbb{R}^{n \times n}$. Since the sparse representation step in variable $ X $ is computed in the same manner as in the works of~\cite{6339108, PRI24}, we consider the isolated transform update step in the alternating minimization algorithm, namely the following optimization problem:
\begin{equation}
\min_{T \Phi \in \mathbb{R}^{n \times n}} \quad \| T \Phi Y - X \|_F^2 \quad \text{s.t.} \quad \|T \|_0 \leq R, \hspace{14pt} \kappa(T \Phi) \leq \rho, \quad \| T \Phi \|_F = \tau,
\label{eq:doubly_sparse_explicit_intermediar}
\end{equation}
where $ R $ represents the target sparsity of matrix $ T $.

The decision variable $ T \Phi $ can be reduced to $ T $ as $ \Phi $ represents a fixed, conventionally known matrix. Following the same consideration, we introduce the variable $\tilde{Y} = \Phi Y$, as the learnt transform acts only on the spectral components in the domain induced by $ \Phi $. The norm-preserving property of the orthogonal canonical transform also simplifies the constraints imposed on the condition number and the Frobenius norm, not affecting the singular spectrum of the overall transform. The final formulation for the \textit{doubly sparse explicitly conditioned transform learning} problem follows:
\begin{equation}
\min_{T \in \mathbb{R}^{n \times n}} \quad \| T \tilde{Y} - X \|_F^2 \quad \text{s.t.}  \quad \|T \|_0 \leq R, \hspace{14pt} \kappa(T) \leq \rho, \quad \| T \|_F = \tau.
\label{eq:doubly_sparse_explicit}
\end{equation}

An optimal solution for the transform update step cannot be derived under both the sparsity and conditioning constraints. The connection between matrix components and singular values has been extensively studied in the literature, especially in \textit{perturbation theory}, without conclusive results for this particular problem. Thus, we consider the relaxed surrogate of the NP-hard formulation, substituting the $ \ell_0 $ pseudo-norm constraint in Problem~\eqref{eq:doubly_sparse_explicit} with the convex $ \ell_1 $-norm penalty function: 
\begin{equation}
\min_{T \in \mathbb{R}^{n \times n}} \quad \| T \tilde{Y} - X \|_F^2 + \lambda \| \text{vec}(T) \|_1 \quad \text{s.t.} \quad \kappa(T) \leq \rho, \quad \| T \|_F = \tau.
\label{eq:doubly_sparse_explicit_relaxed}
\end{equation}

We proceed with a theoretical analysis of the challenges that arise in Problem~\eqref{eq:doubly_sparse_explicit_relaxed}.

%% file: sections/2-theoretical_analysis.tex
% !TEX root = ./main.tex

\section{Theoretical Analysis}

In the literature, optimization problems that consist of a smooth objective and a non-smooth penalty are typically addressed using Forward-Backward splitting methods. Specifically, when the 'forward' step evaluates the gradient of the smooth term, this framework is widely known as Proximal Gradient Descent. When the penalty is chosen to be the $ \ell_1 $-norm of the decision variable, the specific scheme is classically referred to as the Iterative Shrinkage-Thresholding Algorithm (ISTA) and yields the following update for any iterate $ T_k $:
\begin{equation}
    T_{k+1} = \mathcal{S}_{\alpha \lambda}(T_k - \alpha \nabla f(T_k)),
\end{equation}
where $ f $ denotes the smooth objective, $ \alpha $ is the step size, $ \lambda $ is the regularization parameter and $ \mathcal{S}_{\alpha \lambda} $ represents the soft-thresholding operator. Monotone convergence to a global optimum is guaranteed if $ f $ is convex.

The overall objective function in Problem~\eqref{eq:doubly_sparse_explicit_relaxed} comprises a smooth, convex, quadratic objective and the non-smooth, convex $ \ell_1 $-norm penalty, which makes it an ideal candidate for ISTA. However, the additional constraints imposed on the singular spectrum define a highly non-convex feasible space. While this non-convexity precludes guarantees of finding the global minimum, we discuss below how the algorithm's convergence can still be analyzed within modern frameworks for non-convex optimization.

\subsection{Singular Spectrum Projection}
We define the projection operator $ P_{C(\rho, \tau)}(T) $ as the function that minimizes the distance between any given $ T \in \mathbb{R}^{n \times n} $ and the space $ C(\rho, \tau) = \{ A \in \mathbb{R}^{n \times n} \mid \kappa(A) \leq \rho, \| A \|_F = \tau\}$ in the Frobenius norm sense. The following lemma introduces a crucial result for both the convergence and the implementation of the proposed algorithm.

\begin{lemma}[Unique Closed-Form Projection]
    \label{lem:unique_proj}
    Let $ T \in \mathbb{R}^{n \times n}$ with its Singular Value Decomposition given by $ T = U \Sigma V^T $, where $ \Sigma = \text{diag}(\sigma) $ and $ \sigma \in \mathbb{R}^n_+ $ denotes the vector of singular values and let $ \rho \geq 1 $ and $ \tau > 0 $ be given real constants. The solution $ \hat{T} $ to the matrix nearness problem:
    \begin{equation}
        \arg\min_{\hat{T} \in \mathbb{R}^{n \times n}} \quad \| \hat{T} - T \|_F^2  \quad \text{s.t.} \quad \kappa(\hat{T}) \leq \rho, \quad \| \hat{T} \|_F = \tau
        \label{eq:matrix_nearness}
    \end{equation}
    preserves the singular vectors of $ T $, i.e., $ \hat{T} = U \hat{\Sigma} V^T $, and its singular spectrum coincides with the optimum $ \hat{\sigma}$ of the following 1D vector projection problem:
    \begin{equation}
        \arg\min_{\hat{\sigma} \in \mathbb{R}^{n}_+} \quad \| \hat{\sigma} - \sigma \|_2^2  \quad \text{s.t.} \quad \frac{\hat{\sigma}_{\max}}{\hat{\sigma}_{\min}} \leq \rho, \quad \| \hat{\sigma} \|_2 = \tau.
        \label{eq:1D}
    \end{equation}

    Furthermore, the two constraints can be decoupled. Let $ P_{K(\rho)}(\sigma) $ be the projection of $ \sigma $ onto the closed convex cone defined by $ K(\rho) = \{ x \in \mathbb{R}^n_+ \mid x_{\max} \leq \rho x_{\min} \}$. Then $ P_{K(\rho)}(\sigma) $ can be computed in $ \mathcal{O}(n) $ using the geometric approach introduced in~\cite{LI2020190} and the unique global singular values satisfying both constraints are given exactly by scaling the cone projection onto the Frobenius hypersphere:
    \begin{equation}
        \hat{\sigma} = \tau \frac{P_{K(\rho)}(\sigma)}{\| P_{K(\rho)}(\sigma) \|_2}.
    \end{equation}
\end{lemma}

\begin{proof}
    Given the invariance of the Frobenius norm under orthogonal transformations, the singular spectrum constrained Problem~\eqref{eq:matrix_nearness} can be reduced to Problem~\eqref{eq:1D} as a direct consequence of the equality condition of von Neumann's trace inequality. 

    In this form, we can extend the results from Section 2 of~\cite{LI2020190} to the more general case of matrices that are not necessarily symmetric and positive definite, obtaining a unique optimum using Algorithm 2.1 proposed in the same paper in linear time. Because $ K(\rho) $ is a closed convex cone for any $ \rho \geq 1 $ (which is the case of the conditioning parameter in the original Problem~\eqref{eq:doubly_sparse_explicit_relaxed}), the projection $ P_{K(\rho)}(\sigma) $ uniquely determines the optimal direction (or angle) of $ \sigma $, while the projection onto the boundary of the Frobenius ball strictly determines its magnitude (or length). Since any non-zero Euclidean vector $ \sigma \in \mathbb{R}^n_+ $ is uniquely defined by its direction and magnitude and given the fact that these two constraints are orthogonal, i.e., they do not interfere with each other, it follows that the two projections can be fully decoupled, yielding the exact and unique projection $ P_{C(\rho, \tau)}(T) $. 
\end{proof}

\subsection{Convergence Analysis}
Given the interplay between the proximal operator $ \mathcal{S}_{\alpha \lambda} $ and the projection operator $ P_{C(\rho, \tau)} $, a classical Forward-Backward algorithm is not directly tractable. This challenge suggests an extension to the Proximal Alternating Linearized Minimization (PALM) framework, introduced by Bolte et al. in~\cite{bolte:hal-00916090}, which provides a basis for establishing global convergence. Such an approach would naturally integrate into the existing Alternating Minimization (AM) scheme of our original problem, namely the transform update step and the sparse representation step. 

However, for the sake of simplicity and reduced computational complexity, we rely on the uniqueness of the projection derived in Lemma~\ref{lem:unique_proj}. We define the projected version of ISTA:
\begin{equation}
    T_{k+1} = P_{C(\rho, \tau)}(\mathcal{S}_{\alpha \lambda}(T_k - \alpha \nabla f(T_k))),
    \label{eq:iteration}
\end{equation}
and assume that the projection $ P_{C(\rho, \tau)} $ neither negates the decrease achieved by the gradient update, nor does it significantly compromise the sparsity promoted by the proximal operator, particularly as the iteration index $ k $ increases. We note that the sequence of iterates $ T_k $ is intrinsically bounded by the Frobenius norm projection onto the hypersphere of radius $ \tau $. Furthermore, the overall objective function is bounded from below, as the two norms in Problem~\eqref{eq:doubly_sparse_explicit_relaxed} can only take positive values. Finally, the objective function has the Kurdyka-{\L}ojasiewicz (KL) property because it is semi-algebraic, being composed of polynomials and sets defined by polynomial inequalities. 

Under these practical assumptions, we aim to approximate the properties of an ideal joint proximal operator, which is highly intractable due to the sequential application of the non-convex projection. Our approach, thereafter, is strongly motivated by Theorem 5.1 from Attouch et al.~\cite{attouch2011} regarding the inexact Forward-Backward splitting algorithm, which provides a theoretical framework that justifies the conditional convergence to a critical point of the proposed method in practice.

%% file: sections/3-algorithm.tex
% !TEX root = ./main.tex

\section{Proposed Implementation}

To improve the convergence rate of the proposed method, we employ Nesterov's acceleration scheme, thus transforming the projected Iterative Shrinkage-Thresholding Algorithm (ISTA) step into a Fast Iterative Shrinkage-Thresholding Algorithm (FISTA) update, introduced by Beck and Teboulle in~\cite{Beck2009AFI}. However, because our transform update constitutes a single step in the original Alternating Minimization framework, where the objective function varies as the sparse representation $ X $ is updated, relying on a fixed step size bounded by $ 1/L $, where $ L $ denotes the Lipschitz constant, is overly conservative. Instead, we compute the gradient step size $ \alpha_k $ dynamically at each step using the steepest descent approach:
\begin{equation}
    \alpha_k = \frac{1}{2} \frac{\| \nabla f(T_k) \|_F^2}{\| \nabla f(T_k) \tilde{Y} \|_F^2}.
\end{equation}

The sparsity constraint imposed by the penalty function effectively contracts the feasible space, and, by extension, the search space. This has a positive impact on convergence for higher choices of $ \rho $, which correspond to a more permissive feasible space. Thus, we adopt a continuation strategy for the regularization parameter $ \lambda $. A similar strategy is employed in the proximal-gradient homotopy method for sparse least-squares problems proposed in~\cite{10.1137/120869997}. We initialize the algorithm with a relatively large value for the penalty parameter. By starting with a highly constrained problem, the algorithm rapidly isolates the most significant components of the transform, which helps circumvent poor local minima and improves convergence speed. As the iterations progress, we gradually relax the penalty parameter towards its target value. 

Following the singular spectrum projection $ P_{C(\rho, \tau)} $, numerical perturbations are inevitable. Furthermore, while Nesterov's acceleration significantly improves the global convergence rate, its inherent momentum term computes a linear combination of previous iterates. This mechanism is known to induce oscillatory behavior, commonly referred to as the "ripple effect", frequently reviving previously zeroed elements into negligibly small non-zero values. To handle this, a post-projection clipping tolerance is introduced to enforce exact zeros for negligibly small entries. Note that as the target condition number $ \rho $ or the dimensionality $ n $ of the matrix increases, the structure becomes increasingly sensitive to these perturbations. Thus, the clipping tolerance must be strictly controlled and not overly generous. Additionally, we observe empirically that the sparse support of the matrix $ T $ tends to stabilize after a certain number of iterations. Consequently, we introduce a computational heuristic: once the sparsity pattern stabilizes, we fix the non-zero structure of the matrix and substitute the soft-thresholding operator $ \mathcal{S}_{\alpha \lambda} $ with a hard-thresholding step. Note that this heuristic step combined with the relaxation of $ \ell_1 $-norm penalty parameter compensates for the fact that we prioritize the spectral projection over the Soft-Thresholding step.

The overall approach is summarized in Algorithm~\ref{alg:proposed}. As the dimension of the problem grows, the linear time complexity $ \mathcal{O}(n) $ of the 1D projection derived in Lemma~\ref{lem:unique_proj} guarantees that the overall computational bottleneck remains the SVD, resulting in a cubic time complexity. However, matrix operations in the proposed algorithm could benefit from the imposed sparse structure of $ T $, especially as the SVD operation comes after the soft-thresholding step.

\begin{algorithm}[htbp]
\caption{Doubly Sparse Explicitly Conditioned Transform Learning}
\label{alg:proposed}
\begin{algorithmic}[1]
\Require Data matrix $Y$, Analytical orthogonal transform $ \Phi $, target column sparsity $r$, parameters $\rho, \tau$, target penalty $\lambda$, max iterations $M$, stabilization iteration $K$, clipping tolerance $\epsilon$, homotopy steps $N_h$.
\State Initialize $\tilde{Y} = \Phi Y$, $T_0 = I_n$ , $T_{-1} = T_0$, $t_1 = 1$.
\State Generate homotopy sequence $\Lambda = \{\lambda_1, \dots, \lambda_{N_h}\}$ logarithmically decreasing from an initial $\lambda_{start}$ to $\lambda$.
\For{$k = 1, \dots, M$}
    \State Compute Sparse Representation: $X_{k} = H_r(T_{k-1} \tilde{Y})$
    \State Update Nesterov Momentum: $t_{k+1} = \frac{1 + \sqrt{1 + 4t_k^2}}{2}$, \quad $\beta_k = \frac{t_k - 1}{t_{k+1}}$
    \State Compute: $Z_k = T_{k-1} + \beta_k(T_{k-1} - T_{k-2})$
    \State Compute Gradient: $ \nabla f(Z_k) = 2(Z_k \tilde{Y} - X_k)\tilde{Y}^T $
    \State Compute gradient step size: $ \alpha_k = \frac{1}{2} \frac{\|\nabla f(Z_k)\|_F^2}{\|\nabla f(Z_k) \tilde{Y}\|_F^2} $
    \State Gradient descent step: $ \hat{Z}_k = Z_k - \alpha_k \nabla f(Z_k) $
    \If{$k \leq K$}
        \State Select $\lambda_k$ from $\Lambda$ (or $\lambda$ if $k > N_h$)
        \State Apply soft-thresholding: $ \tilde{Z}_k = \mathcal{S}_{\alpha_k \lambda_k}(\hat{Z}_k) $
    \Else
        \State Skip soft-thresholding: $ \tilde{Z}_k = \hat{Z}_k $ 
    \EndIf
    \State Singular Value Decomposition: $ U \Sigma V^T = \tilde{Z}_k $
    \State Extract singular values $ \sigma = \text{diag}(\Sigma) $
    \State Compute 1D Cone Projection: $ \hat{\sigma} = P_{K(\rho)}(\sigma) $ 
    \State Scale to Frobenius hypersphere: $ \hat{\sigma} = \tau \frac{\hat{\sigma}}{\|\hat{\sigma}\|_2} $
    \State Reconstruct projected matrix: $ \tilde{T}_k = U \text{diag}(\hat{\sigma}) V^T $
    
    \If{$k < K$}
        \State Apply clipping tolerance: $T_k = \tilde{T}_k \odot (|\tilde{T}_k| > \epsilon)$
    \ElsIf{$k == K$}
        \State Apply clipping tolerance: $T_k = \tilde{T}_k \odot (|\tilde{T}_k| > \epsilon)$
        \State Extract and save sparsity mask: $\Omega = (T_k \neq 0)$
    \Else
        \State Apply hard-thresholding to fixed support: $T_k = \tilde{T}_k \odot \Omega$
    \EndIf
\EndFor
\State \textbf{return} $T_M, X_M$
\end{algorithmic}
\end{algorithm}

We empirically validate the use of our algorithm in the following section.

%% file: sections/4-numerical_experiments.tex
% !TEX root = ./main.tex

\section{Numerical Experiments}

Proceeding further, the focus will shift to the empirical evaluation of the algorithm introduced in the previous section, compared to existing approaches. The actual implementation is publicly available online\footnote{\url{https://github.com/calcuttarain/Doubly-Sparse-Explicitly-Conditioned-Transform-Learning}} and the code was developed in MATLAB R2026a. For a fair comparison, we used the exact code provided by Pătrașcu et al. in~\cite{PRI24} and Ravishankar and Bresler in~\cite{6466952}. Following a similar testing strategy to that of the aforementioned authors, in the factorization of the transform, $ W = T \Phi $, the canonical transform, $ \Phi $, is taken to be the orthogonal linear operator $ C \otimes C$, where $ C $ denotes the $ \sqrt{n} $ by $ \sqrt{n} $ Discrete Cosine Transform (DCT) matrix. The representation, $ X $, is initially assumed to lie in the domain induced by $ \Phi $ and the learnt transform, $ T $, is initialized with the unconstrained least-squares solution. 

The sparsity level for each column of the sparse representation matrix, $ X_i $, is fixed at $ r = 6 $. The upcoming experiments are performed on $ \sqrt{n} \times \sqrt{n} $ non-overlapping patches extracted from the $ 512 \times 512 $ grayscale images of \textit{Barbara}, \textit{Couple} and \textit{Cameraman}, which are used for training, i.e., for the matrix $ Y $, and of images \textit{Hill} and \textit{Man} which are excluded from the transform learning phase in order to observe generalization performance. Each patch has its mean subtracted, being afterwards vectorized and stacked as a column in the data matrix $ Y $. Given the patch size of $ 8 \times 8 $, the dimension of the resulting vector is $ n = 64 $, and thus, the matrices $ W, T, \Phi \in \mathbb{R}^{64 \times 64} $. 

A comparative analysis will be presented between the algorithms proposed by the authors referenced at the beginning of this section and the algorithm introduced in this paper. For perspective, results obtained using the Discrete Cosine Transform are included. Various values of the parameter $ \lambda $ were employed as the starting point for implementing Bresler's method. The resulting transform yields the parameters $ \rho = \kappa(W) $ and $ \tau = \| W \|_F $, which serve as inputs for computing both the sparse and dense conditioned transform learning algorithms, as seen in Figure~\ref{fig:fig1}. Two metrics are considered in the evaluation of the minimization of the cost function. The first and the most straightforward one is the squared Frobenius norm, i.e., $ \| W Y - X \|_F^2 $. The second metric involves a normalized version, measuring the relative reconstruction error: $ \| W Y - X \|_F^2 / \| W Y \|_F^{2} $.

To further evaluate our approach, we consider an image denoising application, with the results summarized in Table~\ref{tab:denoising_results}. Again, following a similar strategy to that of the authors of~\cite{6466952,7045534,PRI24}, the evaluation is performed under various additive white Gaussian noise levels, for $ \sigma \in \{5, 10, 15, 20, 100\}$. In this scenario, the transform, $ T \in \mathbb{R}^{121 \times 121}$, is trained adaptively on the corrupted image using fully overlapping patches ($r=1$). At each iteration, a random subset of $N=500$ patches is selected for the transform update step. The variable sparsity update step employs an error tolerance threshold scaling factor of $C=1.04$. The alternating minimization is executed for 20 iterations (reduced to 5 iterations for the severe noise case of $\sigma=100$). The homotopy strategy is initiated after 10 iterations, and the fixed support heuristic is omitted in this experiment. The proposed method outperforms both the dense learning variants introduced in~\cite{7045534,PRI24} and the unconditioned doubly sparse method provided in~\cite{6466952} across low to moderate noise regimes. The latter, however, exhibits a lower tendency to overfit on highly corrupted data. 

\begin{figure}[htbp]
    \centering
    
    % Test 1 (Bresler: 2.1e-4)
    \begin{subfigure}{\textwidth}
        \centering
        \includegraphics[width=0.48\linewidth]{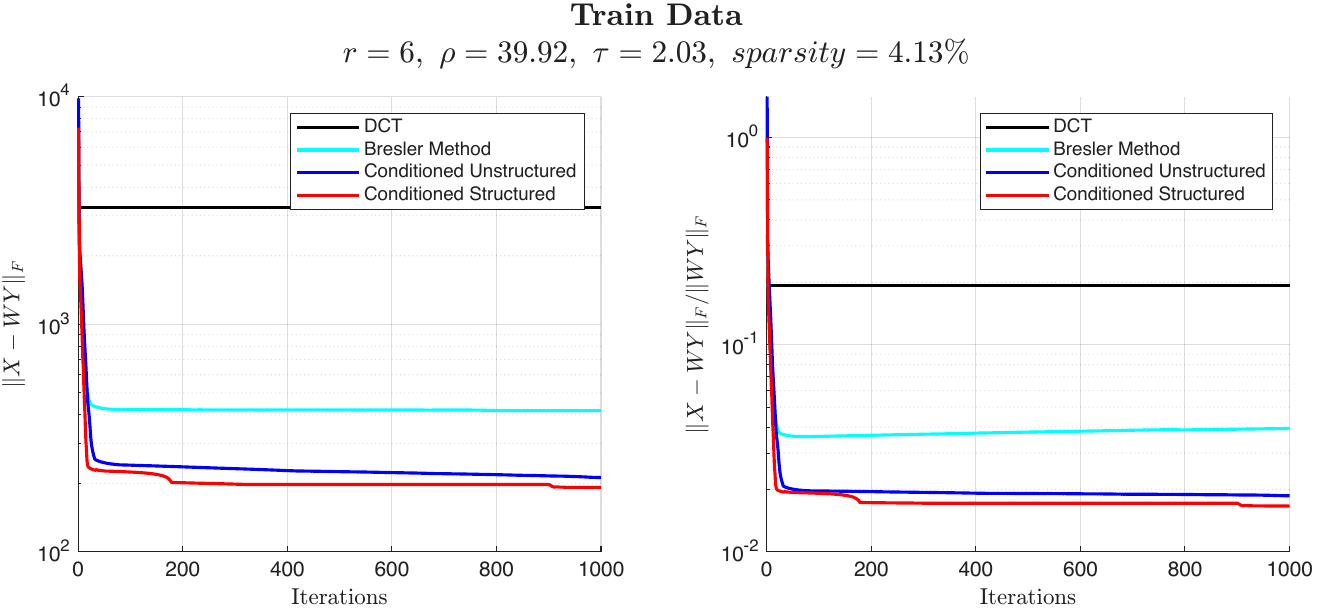}
        \hfill
        \includegraphics[width=0.48\linewidth]{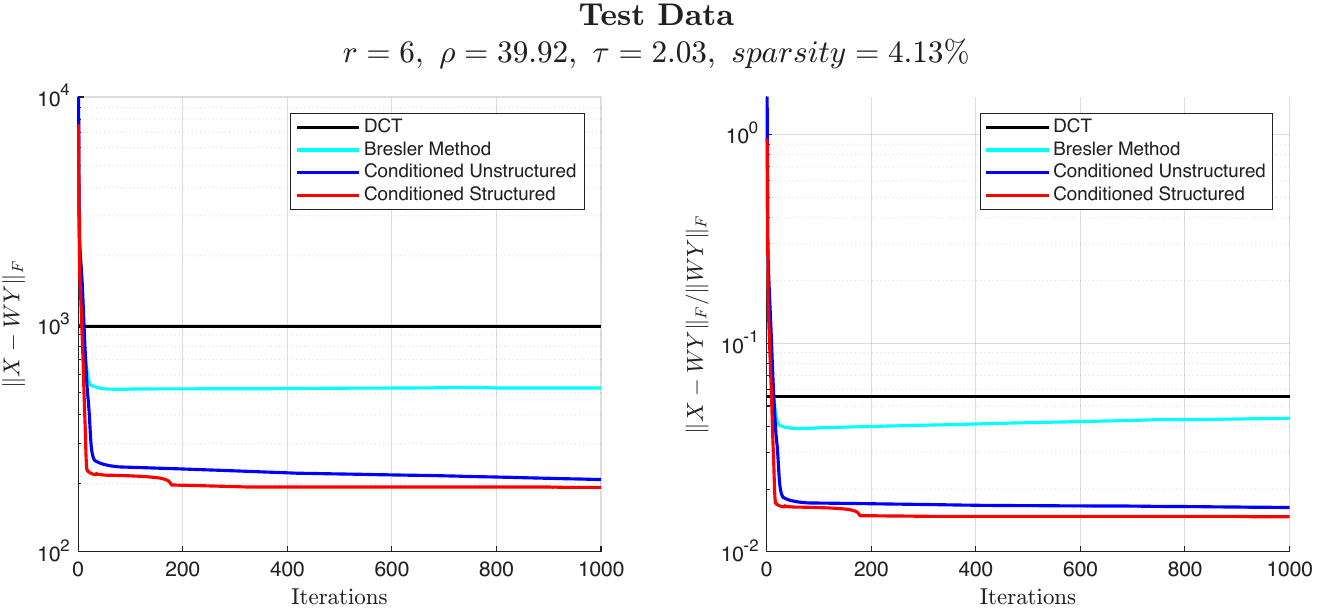}
        \caption{$\lambda_{\text{bresler}} = 2.1 \times 10^{-4}$}
        \label{fig:sub1}
    \end{subfigure}
    
    \vspace{1em} 

    % Test 2 (Bresler: 2.1e-7)
    \begin{subfigure}{\textwidth}
        \centering
        \includegraphics[width=0.48\linewidth]{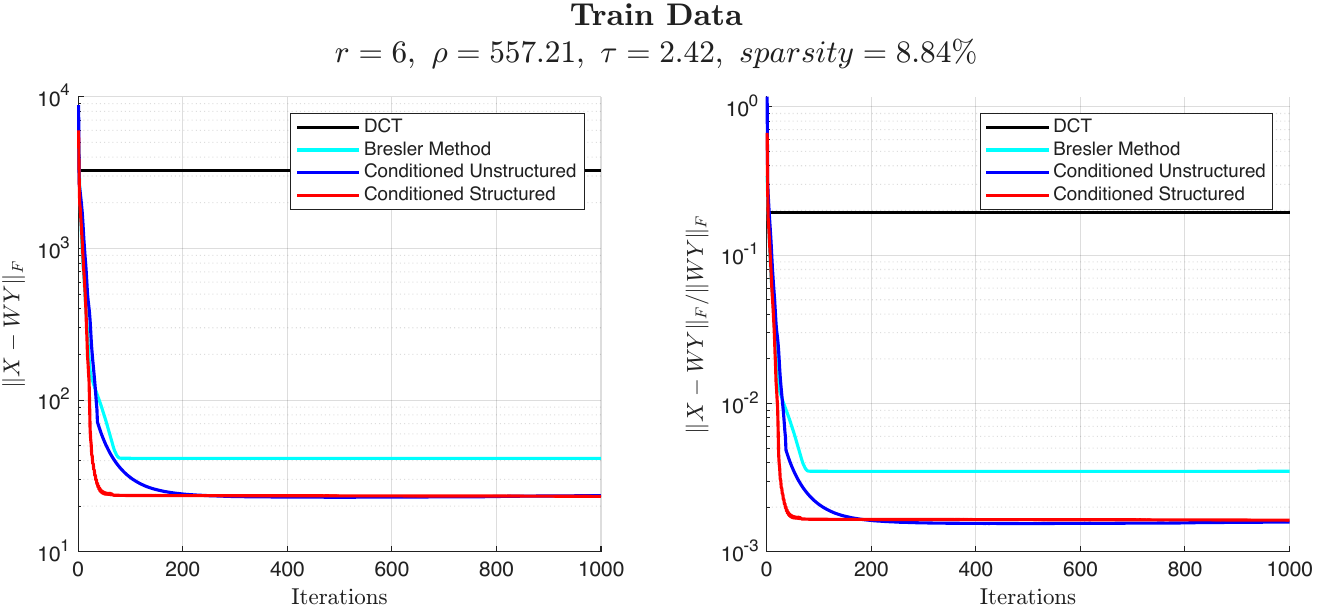}
        \hfill
        \includegraphics[width=0.48\linewidth]{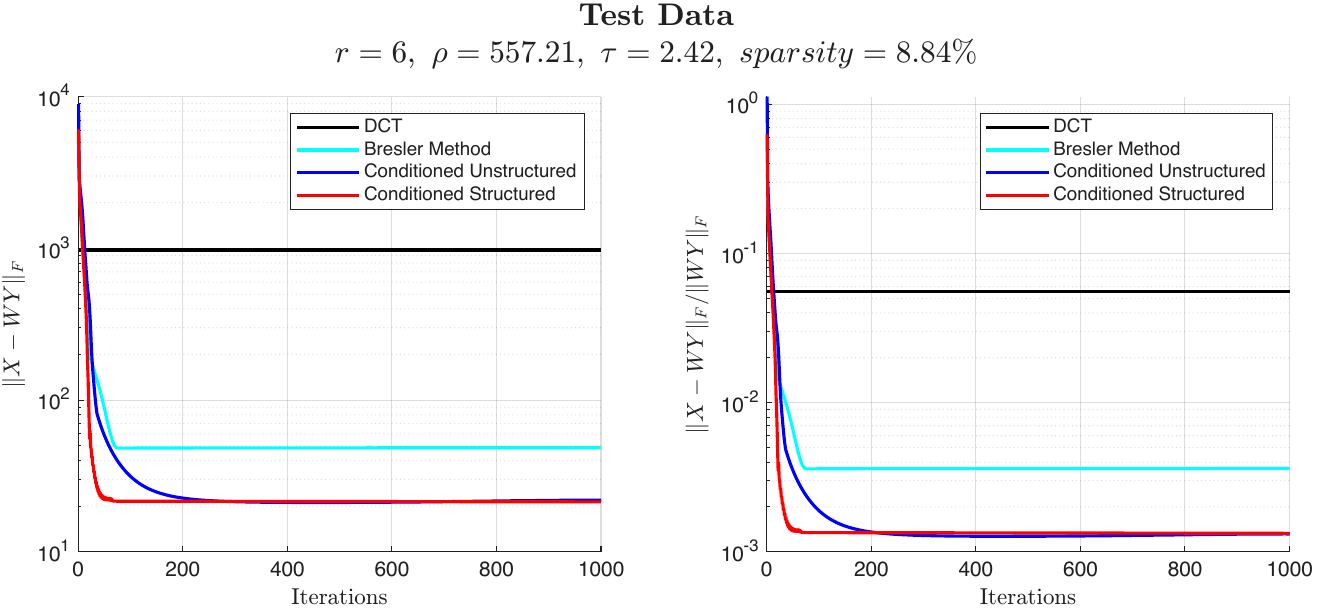}
        \caption{$\lambda_{\text{bresler}} = 2.1 \times 10^{-7}$}
        \label{fig:sub2}
    \end{subfigure}
    
    \vspace{1em}

    % Test 3 (Bresler: 2.1e-12)
    \begin{subfigure}{\textwidth}
        \centering
        \includegraphics[width=0.48\linewidth]{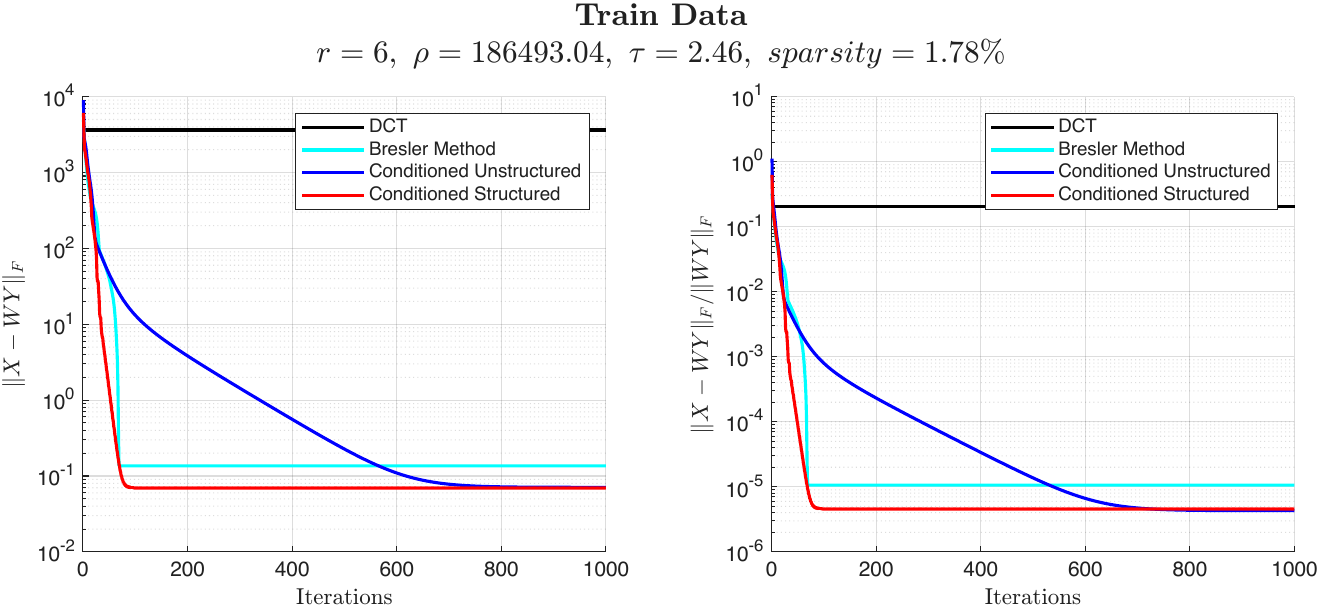}
        \hfill
        \includegraphics[width=0.48\linewidth]{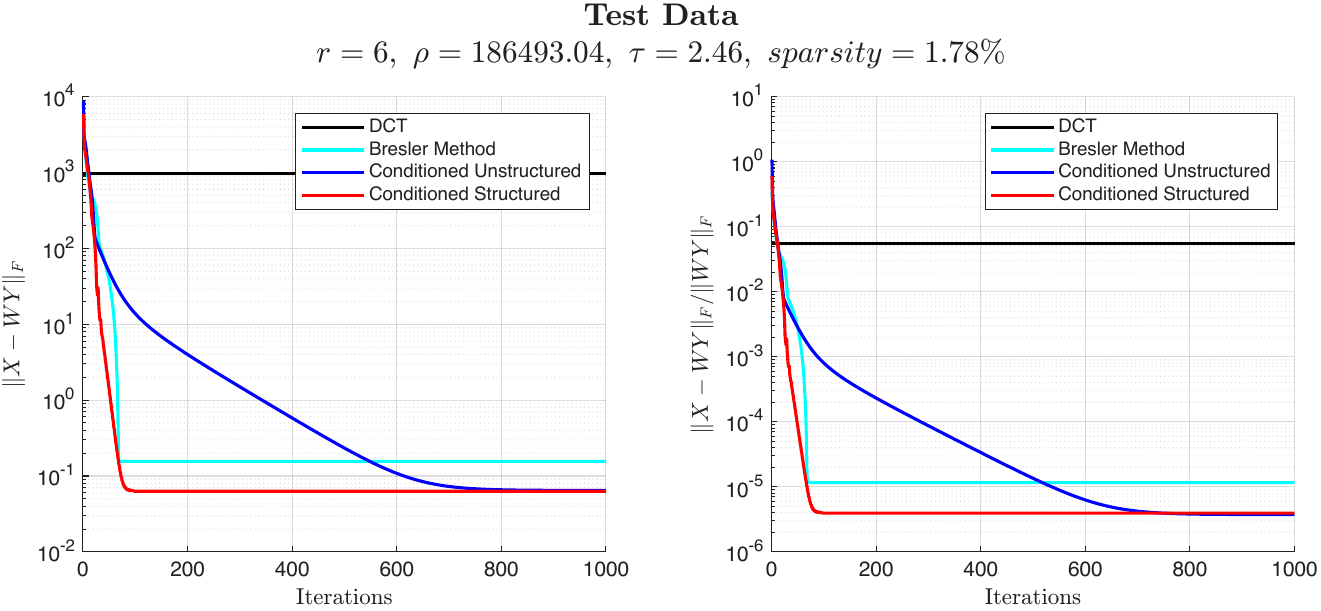}
        \caption{$\lambda_{\text{bresler}} = 2.1 \times 10^{-12}$}
        \label{fig:sub3}
    \end{subfigure}
    
    \caption{Cost function convergence across different conditioning regimes. Black solid line indicates Discrete Cosine Transform performance, cyan indicates the transform learning method provided in~\cite{7045534}, blue indicates the Explicitly Conditioned transform learning of~\cite{PRI24}, and red indicates the proposed method. The \textit{sparsity} value denotes the final sparsity percentage of the proposed transform.}
    \label{fig:fig1}
\end{figure}

\begin{table*}[htbp]
\centering
\caption{Denoising performance comparison in terms of PSNR and SSIM (in paranthesis) across various noise levels ($\sigma$). Best PSNR and SSIM results for each scenario are highlighted in bold.}
\label{tab:denoising_results}
\resizebox{\textwidth}{!}{
\begin{tabular}{llccccc}
\toprule
Image & $\sigma$ & DCT & Dense Bresler & Dense Conditioned & Sparse Bresler & \textbf{Proposed} \\
\midrule
\multirow{5}{*}{\textbf{Cameraman}} & 5 & 37.522 (0.954) & 37.343 (0.953) & 37.335 (0.953) & 37.410 (0.953) & \textbf{37.531} (\textbf{0.954}) \\
 & 10 & 33.309 (0.919) & 33.028 (0.916) & 33.047 (0.916) & 33.119 (0.917) & \textbf{33.329} (\textbf{0.920}) \\
 & 15 & 30.955 (0.884) & 30.521 (0.877) & 30.529 (0.877) & 30.713 (0.881) & \textbf{31.017} (\textbf{0.885}) \\
 & 20 & 29.053 (0.846) & 28.615 (0.836) & 28.632 (0.837) & 28.855 (0.842) & \textbf{29.129} (\textbf{0.846}) \\
 & 100 & \textbf{21.084} (\textbf{0.614}) & 20.419 (0.590) & 20.522 (0.593) & 20.576 (0.590) & 20.081 (0.577) \\
\midrule
\multirow{5}{*}{\textbf{Hill}} & 5 & 36.811 (0.982) & 36.662 (0.981) & 36.663 (0.981) & 36.808 (0.982) & \textbf{36.812} (\textbf{0.982}) \\
 & 10 & 32.968 (0.945) & 32.599 (0.941) & 32.617 (0.941) & 32.872 (\textbf{0.945}) & \textbf{32.973} (0.945) \\
 & 15 & 30.867 (0.902) & 30.433 (0.893) & 30.389 (0.893) & 30.794 (\textbf{0.903}) & \textbf{30.878} (0.902) \\
 & 20 & 29.399 (\textbf{0.856}) & 28.928 (0.843) & 28.934 (0.843) & 29.173 (0.856) & \textbf{29.406} (0.856) \\
 & 100 & \textbf{23.780} (\textbf{0.589}) & 23.654 (0.583) & 23.663 (0.583) & 23.666 (0.586) & 23.660 (0.583) \\
\midrule
\multirow{5}{*}{\textbf{Baboon}} & 5 & 35.057 (0.987) & 35.072 (0.986) & 35.063 (0.986) & \textbf{35.115} (0.986) & 35.058 (\textbf{0.987}) \\
 & 10 & 30.313 (0.956) & 30.114 (0.953) & 30.109 (0.954) & \textbf{30.336} (0.956) & 30.324 (\textbf{0.956}) \\
 & 15 & 27.798 (0.916) & 27.361 (0.908) & 27.382 (0.908) & 27.731 (\textbf{0.917}) & \textbf{27.820} (0.916) \\
 & 20 & 26.099 (0.870) & 25.555 (0.853) & 25.588 (0.854) & 25.961 (\textbf{0.871}) & \textbf{26.131} (0.870) \\
 & 100 & 19.627 (0.409) & 19.564 (0.401) & 19.563 (0.401) & \textbf{19.633} (\textbf{0.417}) & 19.567 (0.402) \\
\midrule
\multirow{5}{*}{\textbf{Barbara}} & 5 & 37.989 (0.987) & 37.689 (0.987) & 37.679 (0.987) & 37.719 (0.987) & \textbf{38.001} (\textbf{0.987}) \\
 & 10 & 34.003 (0.968) & 33.338 (0.967) & 33.373 (0.967) & 33.461 (0.968) & \textbf{34.034} (\textbf{0.968}) \\
 & 15 & 31.609 (0.948) & 30.699 (0.944) & 30.711 (0.943) & 30.907 (0.947) & \textbf{31.671} (\textbf{0.948}) \\
 & 20 & 29.903 (0.924) & 28.890 (0.915) & 28.892 (0.915) & 29.173 (0.922) & \textbf{29.972} (\textbf{0.925}) \\
 & 100 & \textbf{21.521} (\textbf{0.626}) & 21.210 (0.608) & 21.209 (0.608) & 21.290 (0.618) & 21.224 (0.609) \\
\midrule
\bottomrule
\end{tabular}
}
\end{table*}

%% file: sections/5-conclusions.tex
% !TEX root = ./main.tex

\section{Conclusions}

This paper provides the theoretical foundation for the formulation of the \textit{doubly sparse explicitly conditioned transform learning} problem. According to our literature review, this appears to be the first documented instance of this problem. The proposed algorithm shows promising results when compared to existing work, at a lower computational complexity needed for the data reconstruction process. The doubly sparse transform formulation proposed by Ravishankar and Bresler~\cite{6466952} was not included in the cost function convergence test, as it was observed that it does not perform significantly better than the dense one provided in~\cite{7045534}. 

Although global convergence has not been established, empirical observations show that the assumptions made in Section 3 are justified, as the unique projection assures only minimal compromise over the matrix sparsity. The newly derived projection in Lemma~\ref{lem:unique_proj} step is both cheap and exact. Furthermore, the reduced complexity afforded by the sparse structure of $ T $ becomes increasingly advantageous as the problem dimension grows, enabling faster convergence and lowering the time complexity of the rather expensive SVD of dense matrices.

Considering the nature of the signal of interest, one may choose a corresponding canonical transform. An interesting use of the reduced complexity provided by the assumed sparse structure of the transform is the mini-batch processing, in big data applications, and online learning. The latter is particularly valuable in sequentially arriving signals, which require fast processing and efficient storage.

We conclude that the newly introduced product of a mathematically defined transform and an explicitly conditioned learnt factor following a transform-based prior enhances a controlled data adaptation backed by stability and computational efficiency.